\documentclass[draft]{amsart}
\usepackage{a4wide}
\usepackage{amsfonts}
\usepackage{amssymb}
\usepackage{latexsym}
\usepackage{color}
\usepackage{euscript}
\usepackage{mathrsfs} 
\usepackage{dsfont}
\usepackage{bbold}
\usepackage[T1]{fontenc}
\usepackage{setspace}
\usepackage{pgf,tikz}
\usetikzlibrary{arrows}
\definecolor{qqqqcc}{rgb}{0.0,0.0,0.8}
\definecolor{qqqqff}{rgb}{0.0,0.0,1.0}
\definecolor{uuuuuu}{rgb}{0.26666666666666666,0.26666666666666666,0.26666666666666666}

\makeatletter
\@namedef{subjclassname@2010}{%
\textup{2010} Mathematics Subject Classification}
\makeatother

\theoremstyle{plain}
\newtheorem{thm}{Theorem}
\newtheorem{prop}[thm]{Proposition}
\newtheorem{cor}[thm]{Corollary}
\newtheorem{lem}[thm]{Lemma}

\newtheorem{prob}[thm]{Problem}
\theoremstyle{definition}

\newtheorem{exa}[thm]{{\it Example}}
\theoremstyle{remark}
\newtheorem{rem}[thm]{{\it Remark}}
\newtheorem*{rem*}{{\it Remark}}

\DeclareMathOperator{\dess}{{\mathsf{Des}}}
\DeclareMathOperator{\dzii}{{\mathsf{Chi}}}

\DeclareMathOperator{\koo}{{\mathsf{root}}}

\DeclareMathOperator{\paa}{{\mathsf{par}}}

\newcommand*{\card}[1]{\mathrm{card}(#1)}

\newcommand*{\desn}[2]{{\dess^{\langle#1\rangle}(#2)}}
\newcommand*{\des}[1]{{\dess(#1)}}
\newcommand*{\dzD}[1]{{\EuScript D}\big(#1\big)}
\newcommand*{\dz}[1]{{\EuScript D}(#1)}
\newcommand*{\dzi}[1]{\dzii(#1)}
\newcommand*{\dzin}[2]{\dzii^{\langle#1\rangle}(#2)}

\newcommand*{\lambdab}{{\boldsymbol\lambda}}
\newcommand*{\mub}{{\boldsymbol\mu}}
\newcommand*{\betab}{{\boldsymbol\beta}}

\newcommand*{\mphi}{M_{\hat\varphi}}
\newcommand*{\mpsi}{M_{\hat\psi}}
\newcommand*{\pa}[1]{\paa(#1)}

\newcommand*{\slam}{S_{\boldsymbol \lambda}}

\newcommand*{\smalloplus}{\raise0pt\hbox{$\scriptscriptstyle \oplus$}}
\newcommand*{\tcal}{{\mathscr T}}
\newcommand*{\gcal}{{\mathscr S}}

\newcommand*{\wn}[1]{\operatorname{int}(#1)}
\newcommand*{\pth}{\mathscr{P}}
\newcommand*{\multib}{\mathcal{M}(\tcal_\betab,\lambdab)}

\newcommand*{\Vsf}{\mathsf{V}}
\newcommand*{\bpe}[1]{\mathsf{bpe}{(#1)}}
\def\pp{\EuScript{P}}

\newcommand{\Le}{\leqslant}
\newcommand{\Ge}{\geqslant}

\def\is#1#2{\langle#1,#2\rangle}

\def\N{\mathbb N}

\def\R{\mathbb{R}}
\def\C{\mathbb{C}}

\def\ee{\EuScript E}
\def\ff{\EuScript F}
\def\hh{\EuScript H}

\def\bsb{{\mathbf B}}

\DeclareMathAlphabet{\mathpzc}{OT1}{pzc}{m}{it}
\title[Analytic structure of weighted shifts on directed trees]{Analytic structure of weighted shifts on directed trees}
\author{P.\ Budzy{\'n}ski}
\author{P.\ Dymek}
\author{M.\ Ptak}
\address{Katedra Zastosowa\'n Matematyki, Uniwersytet Rolniczy w Krakowie, ul. Balicka 253c, 30-198 Krak\'ow, Poland}
    \email{piotr.budzynski@ur.krakow.pl}
    \email{piotr.dymek@ur.krakow.pl}
    \email{rmptak@cyf-kr.edu.pl}
\thanks{The research was supported by the Ministry of Science and Higher Education of Republic of Poland}
   
\begin{document}
\setstretch{1.2}

\begin{abstract}
We show that a weighted shift on a directed tree is related to a multiplier algebra of coefficients of analytic functions. We use this relation to study spectral properties of the operators in question.
\end{abstract}
\maketitle
\section{Introduction}
Theory of analytic functions plays a central role in operator theory. It has been a source of methods, examples and problems, and has led to numerous important results. The study of the unilateral shift owes much of its success to use of the Hardy space methods (see the monograph \cite{Nikolskii}). The same applies to Toeplitz and Hankel operators or composition operators. Weighted shifts have also been studied with analytic function theory approach. An excellent exposition of an interplay between weighted shift operators and analytic functions has been given by A. Shields in \cite{Shields}. Essential ingredients of the considerations therein were viewing a weighted shift operator as ``multipliation by $z$'' on a Hilbert space consisting of formal power series and showing that the structure of this space is in fact analytic. This enabled using multiplication operators and bounded point evaluations, tools known to be very powerful in variety of problems in operator theory. Other notable examples of using analytic models of operators can be found in \cite{1969-gel, 79-jew-lub, 1983-xia, 1989-sto-sza, 2008-jun-sto}.

In the present paper we study a class of (bounded) weighted shifts on directed trees focusing on analytic aspects of their theory. The class generalizes classical weighted shifts (see \cite{2012-j-j-s-mams}) and is a source of interesting examples (see e.g., \cite{2014-b-d-j-s-aaa,2014-b-j-j-s-jmaa,2012-j-j-s-jfa,2013-j-j-s-sm,2014-j-j-s-ams,2014-j-j-s-ieot,Pietrzycki,Trepkowski}). We start our investigations by introducing the notion of a weighted shift on a weighted directed tree (see Section \ref{WSsec}). Although it is formally a generalization of a weighted shift on a (non-weighted) directed tree, both the notions are equivalent in a sense (see Proposition \ref{thm-unit+-} and Remark \ref{unitarne}). Nevertheless, the technical side of our considerations seems to be easier to handle in the ``weighted directed tree'' setting (in particular, this enables us to consider shifts with weights summing over children of every vertex up to 1, see condition \eqref{stand3}, which is crucial to our study). Then, we define and study multiplier algebras related to weighted shifts on weighted directed trees (see Section \ref{MOsec}). We show that these algebras consists of coefficients of analytic functions (see Propositions \ref{altbpe8} and \ref{analityczna}). Later, we introduce bounded point evaluations and study spectral properties of the adjoints of weighted shifts on weighted directed trees (see Section \ref{BPEsec}). The main result here provides a kind of functional calculus (see Theorem \ref{main}) for functions from multiplier algebras. Next, we investigate the point spectrum of the adjoint of a weighted shift on a weighted directed tree by looking at its behaviour on paths. We show that the spectrum contains a complex disc of radius which can be estimated by the supremum of appropriate limits along the paths (see Theorem \ref{r2+widmo}). Moreover, we show that for particular classes of directed trees the closures of the spectrum and the disc are equal. We conclude the paper with results concerning weighted shifts on (non-weighted) directed trees (see Section \ref{Last}).

\section{Preliminaries}
Let $\N$, $\R$ and $\C$ denote the set of all natural numbers, real numbers and complex numbers, respectively. Set $\N_0=\N\cup\{0\}$ and $\R_+=[0,\infty)$. For $\kappa\in\N\cup \{\infty\}$, $J_\kappa$ stands for the set $\{n\in\N\colon n\Le \kappa\}$. If $Y$ is a set, then $\card{Y}$ denotes the cardinal number of $Y$. For any $r\in(0,\infty)$, $\varDelta_r$ stands for $\{z\in\C\colon |z|<r\}$.

Let $V$ be a nonempty set and $\betab=\{\beta_v\}_{v\in V}$ be a family of positive real numbers. Then $\ell^2(\betab)=\ell^2(V,\betab)$ denotes the Hilbert space of all functions $f\colon V\to\C$ such that $\sum_{v\in V}|f(v)|^2\beta_v<\infty$ with the inner product given by $\is{f}{g}_\betab=\sum_{v\in V} f(v)\overline{g(v)} \beta_v$ for $f,g\in \ell^2(\betab)$. The norm induced by $\is{\cdot}{-}_\betab$ is denoted by $\| \cdot \|_{\beta}$. For $u \in V$, we define $e_u \in \ell^2(\betab)$ to be the characteristic function of the one-point set $\{u\}$; clearly, $\{e_u\}_{u\in V}$ is an orthogonal basis of $\ell^2(\betab)$. We will denote by $\ee$ the linear span of $\{e_u\}_{u\in V}$. Given a subset $V^\prime$ of $V$, $\ell^2(V^\prime, \betab)$ stands for the subspace of $\ell^2(\betab)$ composed of all functions $f$ such that $f(v)=0$ for all $v\in V\setminus V^\prime$, and $\ee_{V^\prime}$ denotes the set of all $f\in \ell^2(V^\prime,\betab)$ such that $\{v\in V\colon f(v)\neq 0\}$ is finite. By $\mathcal{Q}_{V^\prime}$ we denote the orthogonal projection from $\ell^2(\betab)$ onto $\ell^2(V^\prime,\betab)$.

Let $A$ be a (linear) operator in a (complex) Hilbert space $\hh$. Then $\dz{A}$, $r(A)$ and $A^*$ denote the domain, the spectral radius and the adjoint of $A$, respectively (in case it exists). A linear subspace $\ff$ of $\dz{A}$ is called a core of $A$ if $\ff$ is dense in $\dz{A}$ in the graph norm induced by $A$, i.e, the norm $\|\cdot\|_A$ given by $\|f\|_A^2=\|Af\|^2+\|f\|^2$, for $f\in\dz{A}$. If $\ff$ is a subspace of $\hh$, then $A|_\ff$ is the operator in $\hh$ acting on the domain $\dz{A|_\ff}=\ff\cap \dz{A}$ according to the formula $A|_\ff f= Af$. The algebra of all bounded operators on $\hh$ is denoted by $\bsb(\hh)$. If $A \in \bsb(\hh)$, then we denote by $\sigma(A)$ and $\sigma_p(A)$ the spectrum and point spectrum of $A$.
\section{Weighted shifts on directed trees}\label{WSsec}
Let $\tcal=(V,E)$ be a directed tree ($V$ and $E$ stand for the sets of vertices and directed edges of $\tcal$, respectively). Set $\dzi u= \{v\in V\colon (u,v)\in E\}$ for $u \in V$.
Denote by $\paa$ the partial function from $V$ to $V$ which assigns to a vertex $u\in V$ its parent $\pa{u}$ (i.e.\ a unique $v \in V$ such that $(v,u)\in E$). For $k\in \N$, $\paa^k$ denotes the $k$-fold composition of the partial function $\paa$; $\paa^0$ denotes the identity map on $V$. A vertex $u \in V$ is called a {\em root} of $\tcal$ if $u$ has no parent. A root is unique (provided it exists); we denote it by $\koo$. The tree $\tcal$ is {\em rooted} if the root exists. The tree $\tcal$ is {\em leafless} if every vertex $v\in V$ has a child, i.e., $\card{\dzi{v}}\Ge 1$. All the directed trees considered here are assumed to be rooted and countably infinite. In most of the cases they will be also leafless. We set $V^\circ=V\setminus \{\koo\}$. If $v\in V$, then $|v|$ denotes the unique $k\in\N_0$ such that $\paa^k(v)=\koo$. For given $u \in V$ and $n\in\N$ we denote by $\desn{n}{u}$ the set $\big\{v\in V\colon \paa^k(v)=u \text{ for some }k\in J_n\cup\{0\}\big\}$; in addition, we set $\des{u} = \{v\in V\colon \paa^k(v)=u\text{ for some } k\in\N_0\}$.
In the paper we will also consider subgraphs of $\tcal$. On this occasion we will use the notions of children $\dzii$, parent function $\paa$, the root $\koo$, etc., {\em with respect to a subgraph}, say $\mathcal{G}$; they will be denoted by $\dzii_\mathcal{G}$, $\paa_\mathcal{G}$, $\koo_\mathcal{G}$, etc.

Let $\tcal=(V,E)$ be a directed tree. A subgraph $\gcal$ of $\tcal$ which is a directed tree itself is called a {\em subtree} of $\tcal$. A {\em path} in $\tcal$ is a subtree $\pth=(V^\prime,E^\prime)$ of $\tcal$ which satisfies the following two conditions: (i) $\koo\in\pth$, (ii) for every $v\in V^\prime$, $\card{\mathsf{Chi}_{\pth}(v)}=1$. The collection of all paths in $\tcal$ is denoted by $\pp=\pp(\tcal)$. Throughout the paper $\mathbb{1}=\mathbb{1}_V$ stands for either of the families $\{\beta_v\}_{v\in V}$ or $\{\beta_v\}_{v\in V^\circ}$ with $\beta_v=1$ for every $v\in V$. We refer the reader to \cite{2012-j-j-s-mams} for more on directed trees.

Now, we give a definition of a weighted shift on a weighted directed tree. Formally the notion generalizes that of a weighted shift on a directed tree from \cite{2012-j-j-s-mams} but in view of Proposition \ref{thm-unit+-} and Remark \ref{unitarne} (see also Proposition \ref{unilat}) both the notions are equivalent. Let $\tcal=(V,E)$ be a directed tree and let $\lambdab=\{\lambda_v\}_{v \in V^{\circ}} \subseteq \C$. We define then the map  $\varLambda_\tcal^\lambdab\colon \C^V\to \C^V$ via
   \begin{align*} 
(\varLambda_\tcal^\lambdab f) (v) =
   \begin{cases}
\lambda_v \cdot f\big(\pa v\big) & \text{ if } v\in V^\circ,
   \\
0 & \text{ if } v=\koo.
   \end{cases}
   \end{align*}
Given $\betab=\{\beta_v\}_{v \in V} \subseteq \C$, we denote by $\tcal_\betab$ the pair $(\tcal,\betab)$; this is the {\em weighted tree} mentioned above. By a {\em weighted shift on $\tcal_\betab$ with weights} $\lambdab=\{\lambda_v\}_{v \in V^{\circ}} \subseteq \C$,  we mean the operator $\slam$ in $\ell^2(\betab)$ defined as follows
   \begin{align*}
   \begin{aligned}
\dz {\slam} & = \{f \in \ell^2(\betab) \colon \varLambda_\tcal^\lambdab f \in \ell^2(\betab)\},
   \\
\slam f & = \varLambda_\tcal^\lambdab f, \quad f \in \dz{\slam}.
\end{aligned}
\end{align*}
Note that the weighted shift $\slam$ on $\tcal_{\mathbb{1}}$ (in a sense of the above definition) is the {\em weighted shift $\slam$ on the directed tree $\tcal$} as in \cite{2012-j-j-s-mams}.

The questions of boundedness of a $k$--th power of $\slam$ on $\tcal_{\betab}$ and evaluation of its norm can be answered in terms of the products of $k$ consecutive weights. To state the result we need some notation. Let $\tcal=(V,E)$ be a directed tree and let $\lambdab=\{\lambda_v\}_{v \in V^{\circ}} \subseteq \C$. For $u\in V$ and $v\in \dess(u)$ we define
\begin{align*}
\lambda_{u|v}= \lambda_{u|v}^\tcal =
\left\{ \begin{array}{cl}
1 & \text{if } u=v, \\
\prod_{j=0}^{n-1} \lambda_{\paa^j(v)} & \text{if }  v \in \dzii^{\langle n\rangle} (u).
\end{array} \right.
\end{align*}
Arguing as in the proof of \cite[Lemmata 3.1.8 and 6.1.1]{2012-j-j-s-mams}, where $k=1$ and $\betab=\mathbb{1}$ were considered, we get the following.
\begin{lem} \label{norma_tk}
Let $\tcal=(V,E)$ be a directed tree. Let $\lambdab=\{\lambda_v\}_{v \in V^{\circ}} \subseteq \C$ and $\betab=\{\beta_v\}_{v\in V}\subseteq (0,\infty)$. Then the weighted shift $\slam$ on $\tcal_{\betab}$ is bounded on $\ell^2(\betab)$ if and only if
\begin{align*}
\sup_{u\in V} \sum_{v\in\dzi{u}}|\lambda_{v}|^2\, \tfrac{\beta_v}{\beta_u}<\infty.
\end{align*}
Moreover, if $\slam$ is bounded, then
\begin{align*}
\| \slam^k\|^2=\sup_{u\in V} \sum_{v\in\dzin{k}{u}}|\lambda_{u|v}|^2\,\tfrac{\beta_v}{\beta_u}, \quad k\in\N.
\end{align*}
\end{lem}

In the course of our study we will use the formula for the adjoint $\slam^*$ of the weighted shift $\slam$ on $\tcal_\betab$, which is given below. For our purposes it suffices to have it in the bounded case. It can be proofed as in \cite[Proposition 3.4.1(ii)]{2012-j-j-s-mams}. 
\begin{lem}\label{adjoint+invar}
Let $\tcal=(V,E)$ be a directed tree, $\lambdab=\{\lambda_v\}_{v \in V^{\circ}} \subseteq \C$  and let $\betab=\{\beta_v\}_{v\in V}\subseteq(0,\infty)$. If $\slam$ is the weighted shift on $\tcal_\betab$ such that $\slam\in\bsb\big(\ell^2(\betab)\big)$, then its adjoint $\slam^*$ is given by
\begin{align*}
\slam^* e_u=
\left\{
  \begin{array}{ll}
    \lambda_u\frac{\beta_u}{\beta_{\paa(u)}} \, e_{\paa(u)}   & \hbox{if $u\in V^\circ$,} \\
    0 & \hbox{if $u=\koo$.}
  \end{array}
\right.
\end{align*}
\end{lem}
\section{Multiplication operators}\label{MOsec}
Multiplication operators on weighted sequence spaces are helpful when studying ``classical'' weighted shifts (see \cite{Shields}). As it turns out, they can be used also in the context of weighted shifts on directed trees. In this section we define counterparts of classical multiplication operators related to weighted shifts on directed trees.

We note that some of the results here are valid in the more general context of reproducing kernel Hilbert spaces (see Acknowledgments) but, for the sake of consistency with subsequent results, we opt for a direct approach and present them in a specific form (for many interesting facts concerning multiplication operators and multipliers in RKHS we refer the reader to \cite{2000-sza, 2003-sza}).

It will be assumed throughout the section\footnote{In view of Propositions \ref{thm-unit+-} and \ref{multikulti} the assumption $\{\lambda_v\}_{v \in V^\circ}\subseteq  (0,\infty)$ is not restrictive (cf. Remark \ref{smierdzi}).} (if not stated differently).
\begin{align} \label{stand2}\tag{$\star$}
   \begin{minipage}{75ex}
$\tcal=(V,E)$ is a countably infinite rooted and leafless directed tree,\\ $\betab=\{\beta_v\}_{v\in V}\subseteq (0,\infty)$ and $\{\lambda_v\}_{v \in V^\circ}\subseteq  (0,\infty)$.
   \end{minipage}
   \end{align}
We start with defining a multiplication operator. Let $\hat\varphi \colon \N_0\to \C$. Define $\varGamma_{\hat\varphi}^\lambdab\colon \C^V\to \C^V$ by
\begin{align*}
\big(\varGamma_{\hat\varphi}^\lambdab f\big)(v) =\sum_{k=0}^{|v|} \lambda_{\paa^k(v)|v} \, \hat\varphi (k) f(\paa^k(v)),\quad v\in V.
\end{align*}
The {\em multiplication operator} $M_{\hat\varphi}^{\lambdab,\betab}\colon \ell^2(\betab)\supseteq \dzD{M_{\hat\varphi}^{\lambdab,\betab}}\to \ell^2(\betab)$  is given by
   \begin{align*}
   \begin{aligned}
\dzD {M_{\hat\varphi}^{\lambdab,\betab}} & = \big\{f \in \ell^2(\betab) \colon \varGamma_{\hat\varphi}^\lambdab f \in \ell^2(\betab)\big\},
   \\
M_{\hat\varphi}^{\lambdab,\betab} f & = \varGamma_{\hat\varphi}^\lambdab f, \quad f \in \dzD{M_{\hat\varphi}^{\lambdab,\betab}}.
\end{aligned}
\end{align*}
The function $\hat\varphi \colon \N_0\to \C$ is said to be the {\em symbol} of $M_{\hat\varphi}^{\lambdab,\betab}$. If no confusion can arise, we write $\varGamma_{\hat\varphi}$ and $M_{\hat\varphi}$ instead of $\varGamma_{\hat\varphi}^\lambdab$ and $M_{\hat\varphi}^{\lambdab,\betab}$, respectively.

As shown in Lemma \ref{mclosed} below, any multiplication operator $\mphi$ is automatically closed.
\begin{lem}\label{mclosed}
Assume \eqref{stand2}. For every $\hat\varphi\colon \N_0\to \C$, the multiplication operator $\mphi$ is closed.
\end{lem}
\begin{proof}
Let $\{f_{n}\}_{n=1}^{\infty}\subseteq \dz{\mphi}$ and $f,g\in \ell^2(\betab)$ satisfy $\lim_{n\to\infty}f_n=f$ and $\lim_{n\to\infty}\mphi f_n=g$. Then for every $v\in V$, $\lim_{n\to\infty} f_n(v)=f(v)$, which implies that
\begin{align*}
\lim_{n\to\infty}\sum_{k=0}^{|v|} \lambda_{\paa^k(v)|v}\, \hat\varphi (k) f_n(\paa^k(v))=\sum_{k=0}^{|v|} \lambda_{\paa^k(v)|v}\, \hat\varphi (k) f(\paa^k(v)), \quad v \in V.
\end{align*}
On the other hand, we have
\begin{align*}
\lim_{n\to\infty}\sum_{k=0}^{|v|} \lambda_{\paa^k(v)|v}\hat\varphi (k) f_n(\paa^k(v))=\lim_{n\to\infty}\big(\mphi f_n\big) (v)=g(v), \quad v \in V.
\end{align*}
Hence, $\varGamma_{\hat\varphi} f=g$. Thus $f\in \dz{\mphi}$ and $\mphi f=g$, which completes the proof.
\end{proof}
Let $\multib$ denote the collection of all functions $\hat\varphi\colon \N_0\to \C$ such that $\dz{\mphi}= \ell^2(\betab)$. Clearly, $\multib$ is a linear space (with standard addition of functions).
\begin{lem}
Assume \eqref{stand2}. For every $\hat\varphi\in \multib$, the multiplication operator $\mphi$ is bounded on $\ell^2(\betab)$.
\end{lem}
\begin{proof}
It is well-known that every closed everywhere defined Banach space operator is automatically bounded (\cite[Theorem III.12.6]{Conway-fa}). Thus, we deduce boundedness of $\mphi$ from Lemma \ref{mclosed}.
\end{proof}
In view of the above, $\multib$ is a normed space with the norm $\|\cdot\|$ given by
\begin{align*}
\|\hat\varphi\|\overset{df}{=}\|\mphi\|,\quad \hat\varphi\in \multib.
\end{align*}
The space $\multib$ turns out to have a natural Banach algebra structure. It suffices to endow it with the Cauchy-type multiplication $*\colon \C^{\N_0}\times\C^{\N_0}\to \C^{\N_0}$ given by
\begin{align}\label{cauchymult}
\big(\hat\varphi*\hat\psi\big) (k) = \sum_{j=0}^{k} \hat\varphi(j)\hat\psi(k-j),\quad \hat\varphi, \hat\psi\in \C^{\N_0}.
\end{align}
\begin{thm}\label{cauchym}
Let $\tcal=(V,E)$ be a countably infinite rooted and leafless directed tree, $\betab=\{\beta_v\}_{v\in V}\subseteq (0,\infty)$ and $\{\lambda_v\}_{v \in V^\circ}\subseteq (0,\infty)$. Let $\slam\in\bsb\big(\ell^2(\betab)\big)$ be a weighted shift on $\tcal_\betab$. Then following assertions are satisfied$:$
\begin{enumerate}
\item[(i)] For every $n\in\N_0$, $M_{\chi_{\{n\}}}=\slam^n$.
\item[({ii})] If $\hat\varphi\colon \N_0\to \C$ has finite support, then $\hat\varphi\in \multib$.
\item[(iii)] For all $\hat\varphi, \hat\psi\in \multib$, the function $\hat\varphi*\hat\psi$ belongs to $\multib$ and
\begin{align*}
\mphi M_{\hat\psi}= M_{\hat \varphi*\hat\psi}.
\end{align*}
\item[(iv)] The space $\multib$, endowed with 
the Cauchy-type multiplication, is a commutative Banach algebra with unit.
\end{enumerate}
\end{thm}
\begin{proof} (i) This follows directly from the definitions of $\mphi$ and $\slam$.

(ii) This follows from (i) and linearity.

(iii) We first show that
\begin{align}\label{gammamnozenie}
\varGamma_{\hat\varphi} \varGamma_{\hat\psi} f= \varGamma_{\hat \varphi*\hat\psi} f, \quad f\in \C^V \text{ and } \hat\varphi,\hat\psi\in \C^{\N_0}.
\end{align}
Fix $\hat\varphi,\hat\psi\in \C^V$ and take $f \in \C^V$. Note that for every $v\in V$ and all $k\in \{0,\ldots, |v|\}$ and $j\in \{0,\ldots, |\paa^k(v)|\}$, the expression $\paa^j(\paa^k(v))$ makes sense and equals $\paa^{j+k}(v)$; moreover $|\paa^k(v)|=|v|-k$. Hence, we have the equality\allowdisplaybreaks
\begin{align}\label{mnozenie1}
\big(\varGamma_{\hat\varphi} \varGamma_{\hat\psi} f\big) (v)&=\sum_{k=0}^{|v|} \lambda_{\paa^k(v)|v}\, \hat\varphi (k)\big(\varGamma_{\hat\psi} f \big)(\paa^k(v))\notag\\
&=\sum_{k=0}^{|v|} \lambda_{\paa^k(v)|v}\,\hat\varphi (k) \sum_{j=0}^{|\paa^k(v)|} \lambda_{\paa^j(\paa^k(v))|\paa^k(v)}\, \hat\psi(j)\,f(\paa^j\paa^k(v))\notag\\
&=\sum_{k=0}^{|v|}\lambda_{\paa^k(v)|v}\, \hat\varphi (k) \sum_{j=0}^{|v|-k} \lambda_{\paa^{j+k}(v)|\paa^k(v)}\,\hat\psi(j)\,f(\paa^{j+k}(v))\notag\\
&=\sum_{k=0}^{|v|}  \sum_{j=k}^{|v|} \lambda_{\paa^k(v)|v}\, \lambda_{\paa^{j}(v)|\paa^k(v)}\  \hat\varphi (k)\, \hat\psi(j-k)\,f(\paa^{j}(v)),\notag \\
&=\sum_{k=0}^{|v|} \sum_{j=k}^{|v|} \lambda_{\paa^{j}(v)|v}\ \hat\varphi (k) \, \hat\psi(j-k)\,f(\paa^{j}(v)),\quad v\in V.
\end{align}
On the other hand, we see that
\begin{align}\label{mnozenie2}
\big(\varGamma_{\hat\varphi*\hat\psi} f\big)(v)=\sum_{j=0}^{|v|} \lambda_{\paa^{j}(v)|v}\ \sum_{k=0}^j \hat\varphi (k)\hat\psi (j-k) \  f(\paa^j(v)),\quad v\in V.
\end{align}
By \eqref{mnozenie1} and \eqref{mnozenie2}, we get
\begin{align*}
\big(\varGamma_{\hat\varphi} \varGamma_{\hat\psi} f\big) (v)= \big(\varGamma_{\hat\varphi*\hat\psi} f\big)(v), \quad v\in V.
\end{align*}
Hence $\dz{M_{\hat\varphi *\hat\psi}}=\ell^2(\betab)$ and $\dz{M_{\hat\varphi *\hat\psi}}= \mphi \mpsi$, which proves (iii).

(iv) In view of \eqref{cauchymult}, (i) and (iii), $\multib$ is a commutative algebra with unit. We prove now that $\multib$ is closed. To this end take a Cauchy sequence $\{\hat\varphi_n\}_{n=1}^\infty\subseteq \multib$. Then $\{M_{\hat\varphi_n}\}_{n=1}^\infty$ is a Cauchy sequence in $\bsb(\ell^2(\beta))$, and hence there is $A\in\bsb(\ell^2(\beta))$ such that $\lim_{n\to\infty}M_{\hat\varphi_n}=A$. We note that
\begin{align}\label{niezalezne}
\frac{\big(A e_{\koo}\big)(u_1)}{\lambda_{\koo|u_1}}=\frac{\big(A e_{\koo}\big)(u_2)}{\lambda_{\koo|u_2}},\quad \text{ if }  u_1,u_2\in V \text{ satisfy }|u_1|=|u_2|,
\end{align}
which easily follows from
\begin{align}\label{psi1}
(A e_{\koo}) (v)= \lim_{n\to\infty} \sum_{k=0}^{|v|} \lambda_{\paa^k(v)|v}\ \hat\varphi_n(k) \,e_{\koo} ({\paa}^k(v))
=\lim_{n\to \infty} \lambda_{\koo|v}\ \hat\varphi_n(|v|), \quad v\in V.
\end{align}
Let us define the function $\hat\psi\colon \N_0\to\C$ by
\begin{align}\label{psi2}
\hat\psi (|v|)= \frac{(A e_{\koo}) (v)}{\lambda_{\koo|v}}.
\end{align}
By \eqref{niezalezne}, $\hat\psi$ is well-defined. Moreover, by \eqref{psi1} and \eqref{psi2}, we have
\begin{multline*}
\big(A f\big)(v)=\lim_{n\to\infty} \big( M_{\hat\phi_n} f\big)(v)= \lim_{n\to\infty} \sum_{k=0}^{|v|} \lambda_{\paa^k(v)|v} \ \hat\varphi_n(k) \,f ({\paa}^k(v))\\
= \sum_{k=0}^{|v|} \lambda_{\paa^k(v)|v}\ \hat\psi(k) \,f ({\paa}^k(v))=\big(\varGamma_{\hat\psi} f\big) (v),\quad v\in V.
\end{multline*}
Since $A$ is bounded, $\dz{\mpsi}=\ell^2(\betab)$ and $A=\mpsi$. Thus $\hat\psi \in \multib$. Since, $\|\hat\psi-\hat\varphi_n\|=\|\mpsi-M_{\hat\varphi_n}\|\to 0$ as $n\to\infty$, we get the claim.
\end{proof}
The above justifies calling $\multib$ the {\em multiplier algebra} induced by $\slam$.
\begin{prop}\label{altbpe8}
Assume \eqref{stand2}. Let $\slam\in\bsb(\ell^2(\betab))$ be a weighted shift on $\tcal_\betab$. Then for every $\hat\varphi\in\multib$ and every $w\in\varDelta_{r(\slam)}$ the series $\sum_{n\in \N_0} \hat\varphi(n)\, w^n$ is absolutely convergent.
\end{prop}
\begin{proof}
For $u\in V$ and $k\in \N_0$ let $f_{u,k}\colon V\to \C$ be a function defined by
\begin{align*}
f_{u,k}(v)=
\left\{
  \begin{array}{ll}
    \lambda_{u|v} & \hbox{if $v\in \dzin{k}{u}$,}\\
    0 & \hbox{otherwise.}
  \end{array}
\right.
\end{align*}
Let us notice that by Lemma \ref{norma_tk} we have
\begin{align*}
\sum_{v \in V} \big(f_{u,k}(v)\big)^2 \beta_v= \sum_{v\in \dzin{k}{u}} \lambda_{u|v}^2 \beta_v \leqslant \beta_u\,\|\slam^k\|^2,
\end{align*}
which implies that $f_{u,k} \in \ell^2(\betab)$ and $\| f_{u,k}\| \leqslant \sqrt{\beta_u}\,\|\slam^k\|$. Moreover, we have
\begin{align*}
\big|\hat{\varphi}(k) \sum_{v\in\dzin{k}{u}} \lambda_{u|v}^2  \beta_v\big|
&= \big|\sum_{v\in\dzin{k}{u}} \lambda_{u|v}\, \hat{\varphi}(k)\, f_{u,k}(v) \beta_v\big| \\
&= \big|\langle M_{\hat\varphi} \, e_u,f_{u,k} \rangle_{\betab} \big| \leqslant \sqrt{\beta_u}\,\|M_{\hat\varphi}\|\, \|f_{u,k}\|_{\betab} \leqslant \beta_u\,\|M_{\hat\varphi}\| \,\|\slam^k\|.
\end{align*}
Thus, taking supremum over $u\in V$ and using Lemma \ref{norma_tk}, we get $\|\slam^k\|^2\, |\hat{\varphi}(k)|\leqslant \|M_{\hat\varphi}\|  \|\slam^k\|$. Consequently, $|\hat{\varphi}(k)| \leqslant \|M_{\hat\varphi}\| \|\slam^k\|^{-1}$ for every $k\in\N_0$. This together with Gelfand's formula for the spectral radius gives the claim.
\end{proof}
\begin{prop}\label{analityczna}
Assume \eqref{stand2}. Let $\slam\in\bsb(\ell^2(\betab))$ be a weighted shift on $\tcal_\betab$. Let $\hat\varphi\colon \N_0\to\C$ be such that the series $\sum_{k=0}^\infty \hat\varphi(k)\, z^n$ is convergent for every $z\in\varDelta_{\|\slam\|}$. If the function $\varphi\colon \varDelta_{\|\slam\|}\to\C$ given by $\varphi(z)=\sum_{k=0}^\infty \hat\varphi(k)\, z^n$ is bounded, then $\hat\varphi\in\multib$ and $\|M_{\hat\varphi}\|\Le \|\varphi\|_\infty:=\sup\{|\varphi(z)|\colon |z|<\|\slam\|\}$.
\end{prop}
\begin{proof}
For $k\in\N$ we define functions $\hat\omega_k\colon\N_0\to\C$ and $\omega_k\colon \varDelta_{\|\slam\|}\to \C$ by
\begin{align*}
\hat\omega_k(n)=
\left\{
\begin{array}{ll}
  \frac{k+1-n}{k+1}\,\hat\varphi(n) & \hbox{if $n\Le k+1$,} \\
  0 & \hbox{otherwise.}
\end{array}
\right.
\end{align*}
and
\begin{align*}
\omega_k(z)=\sum_{n=0}^\infty\hat\omega_k (n)\, z^n,\quad
\end{align*}
It is easily seen that
\begin{align*}
\omega_k(z)=\frac{1}{k+1}\sum_{n=0}^k s_n(z),\quad z\in \varDelta_{\|\slam\|},
\end{align*}
where $s_n\colon \varDelta_{\|\slam\|}\to\C$, $n\in\N_0$, are given by $s_n(z)=\sum_{k=0}^n\hat\varphi(k)\, z^k$. Hence by Theorem \ref{cauchym}, the von Neumann inequality and the well-known facts about Ces\`aro means (see \cite[ p.\ 16-24]{Hoffman}) we have
\begin{align*}
\sup_{k\in\N_0}\|M_{\hat\omega_k}\|=\sup_{k\in\N_0}\|\omega_k(\slam)\|\Le \sup\{ \omega_k(z)\colon z\in\varDelta_{\|\slam\|},\, k\in\N_0\}\Le \|\varphi\|_\infty.
\end{align*}
Moreover, for every $f\in \ell^2(\betab)$ and every finite $W\subseteq V$ we have
\begin{align}\label{finW}
\sum_{v\in W} \Big|\big(M_{\hat\omega_k} f\big) (v)\Big|^2\beta_v\Le \sup_{n\in\N_0}\|M_{\hat\omega_n}\|^2\|f\|^2_\betab\Le \|\varphi\|_\infty^2 \|f\|^2_\betab,\quad k\in\N_0.
\end{align}
Since for every $v\in V$, $\lim_{k\to\infty} (M_{\hat\omega_k}f) (v)=(\varGamma_{\hat\varphi} f)(v)$, we infer from \eqref{finW} that
\begin{align*}
\sum_{v\in V} \Big|\big(\varGamma_{\hat\varphi} f\big) (v)\Big|^2\beta_v\Le \|\varphi\|_\infty^2 \|f\|^2_\betab.
\end{align*}
Hence, $\dz{M_{\hat\varphi}}=\ell^2(\betab)$ and thus $\hat\varphi\in\multib$.
\end{proof}
In view of Proposition \ref{analityczna}, if $r(\slam)=\|\slam\|$ and $\hat\varphi\colon\N_0\to\C$ induces the function $\varphi$ which is analytic in $\varDelta_r$ with $r>r(\slam)$, then $\hat\varphi\in\multib$. On the other hand, if $r(\slam)<\|\slam\|$, then one can find an analytic function $\psi$ on $\varDelta_r$, with $r(\slam)<r<\|\slam\|$, without analytic extension onto $\varDelta_{\|\slam\|}$. Proposition \ref{analityczna} cannot be used to verify if the function $\hat\psi$ determined by the coefficients of the series expansion of $\psi$ at $0$ belongs to $\multib$. One can us the following then (this fact can be proved with help of the Riesz functional calculus but we use approach based on the Gelfand formula for the spectral radius).
\begin{prop}
Assume \eqref{stand2}. Let $\slam\in\bsb(\ell^2(\betab))$ be a weighted shift on $\tcal_\betab$ and let $r\in \big(r(\slam),\infty\big)$. If $\hat\varphi\colon \N_0\to\C$ is such that the series $\sum_{k=0}^\infty \hat\varphi(k)\, z^n$ is convergent for every $z\in\varDelta_{r}$, then $\hat\varphi\in\multib$, the series $\sum_{k=0}^\infty \hat\varphi(k)\, \slam^k$ is norm convergent and $M_{\hat\varphi}=\sum_{k=0}^\infty \hat\varphi(k)\, \slam^k$.
\end{prop}
\begin{proof}
Since $r>r(\slam)$, by the Gelfand formula for the spectral radius there exists $\rho\in(0,1)$ and $k_0 \in \N$ such that
\begin{align*}
\frac{\rho}{\sqrt[k]{|\hat\varphi(k)|}} \Ge \sqrt[k]{\|\slam^k\|},\quad  k \Ge k_0,
\end{align*}
which implies that
\begin{align*}
|\hat\varphi(k)|\,\|\slam^k\|\Le \rho^k,\quad k\Ge k_0.
\end{align*}
Hence, the series $\sum_{k=0}^\infty \hat\varphi(k)\, \slam^k$ is norm convergent to some $S\in\bsb\big(\ell^2(\betab)\big)$. Since $\multib$ is a Banach algebra, $S=M_{\hat\psi}$ with some $\hat\psi\colon \N_0\to\C$. Now, evaluating $ \langle M_{\hat\psi} e_{\koo}, e_u \rangle$, $u\in V$, we deduce that $\hat\psi(k)=\hat\varphi(k)$ for every $k\in\N_0$ and consequently $M_{\hat\varphi}=M_{\hat\psi}=\sum_{k=0}^\infty \hat\varphi(k)\, \slam^k$.
\end{proof}
In view of the above the following problem seems to be interesting.
\begin{prob}
Let $\hat\varphi\colon \N_0\to\C$ be such that the series $\sum_{k=0}^\infty \hat\varphi(k)\, z^n$ is convergent for every $z\in\varDelta_{r(\slam)}$ and the function $\varphi\colon \varDelta_{r(\slam)}\to\C$ given by $\varphi(z)=\sum_{k=0}^\infty \hat\varphi(k)\, z^n$ is bounded. Does this imply that $\hat\varphi\in\multib$?
\end{prob}
We finish the section with an auxiliary result, in view of which a multiplication operator with positive symbol can be effectively approximated with multiplication operators with finitely supported symbols.
\begin{prop}
Assume \eqref{stand2}. Let $\hat\varphi=|\hat\varphi|\in\multib$ and let $\hat\varphi^{(n)}\colon \N_0\to\C$,  $n\in\N$, be given by $\hat\varphi^{(n)}(k)=\chi_{\{1,\ldots,n\}}(k)\, \hat\varphi(k)$, $k\in\N$. Then $\hat\varphi^{(m)}\in\multib$ for every $m\in\N$, and $\{M_{\hat\varphi^{(n)}}\}_{n=1}^\infty$ converges to $M_{\hat\varphi}$ in the strong operator topology.
\end{prop}
\begin{proof}
The first part of the claim follows from Theorem \ref{cauchym}. Now, let $f\in \ell^2(\betab)$. For every $v\in V$ we have
\begin{align*}
\big(M_{\hat\varphi^{(n)}}f\big)(v)=\sum_{k=0}^{|v|} \lambda_{\paa^k(v)|v} \, \hat\varphi^{(n)} (k) f(\paa^k(v))\overset{n\to\infty}{\longrightarrow} \sum_{k=0}^{|v|} \lambda_{\paa^k(v)|v} \, \hat\varphi (k) f(\paa^k(v))=\big(M_{\hat\varphi}f\big)(v).
\end{align*}
Since for every $f \in \ell^2(\betab)$, $M_{\hat\varphi}|f| \in \ell^2(\betab)$ and
\begin{align*}
\sum_{v\in V} \big|\big( M_{\hat\varphi^{(n)}}f\big)(v)\big|^2\beta_v\Le \sum_{v\in V} \Big( \sum_{k=0}^{|v|} \lambda_{\paa^k(v)|v} \, \hat\varphi (k) |f(\paa^k(v))| \Big)^2\beta_v = \|M_{\hat\varphi} |f|\|^2,
\end{align*}
the claim follows from the Lebesgue dominated convergence theorem.
\end{proof}
\begin{prob}
Does there exist $\hat\varphi\in\multib$ such that $|\hat\varphi|\notin\multib$?
\end{prob}
\section{Bounded point evaluations}\label{BPEsec}
In this section we define and study bounded point evaluations related to weighted shifts on weighted directed trees. In the classical weighted shift setting they were used successfully in many problems, for example to show that weighted shifts are reflexive (see \cite{Shields}).
\begin{thm}\label{altbpe0}
Suppose $\tcal=(V,E)$ is a countably infinite rooted directed tree and $\betab =\{\beta_v\}_{v\in V}\subseteq (0,\infty)$. Let $w\in\C$. Then the following conditions are equivalent$:$
\begin{enumerate}
\item[(i)] There exists a continuous linear mapping $\Vsf_w\colon \ell^2(\betab) \to \C$ such that
\begin{align}\label{altbpe1}
\Vsf_w(f)= \sum_{v\in V} f(v)\,w^{|v|},\quad  f \in\ee.
\end{align}
\item[(ii)] There exists $c>0$ such that
\begin{align*}
\Big| \sum_{v\in V} f(v)\, w^{|v|} \Big|\Le c\, \|f\|_{\betab},\quad  f \in\ee.
\end{align*}
\item[(iii)] There exists $k_w\in\ell^2(\betab)$ such that
\begin{align}\label{altbpe2}
\is{f}{k_w}_\betab= \sum_{v\in V} f(v)\, w^{|v|},\quad  f \in\ee.
\end{align}
\item[(iv)] $\sum_{v\in V}\beta_v^{-1}|w|^{2|v|}<\infty$.
\end{enumerate}
Moreover, if any of the conditions $($i$)$-$($iv$)$ holds, then both the mapping $\Vsf_w$ and the function $k_w$ are unique,
\begin{align}\label{altbpe2+}
k_w(v)=\frac{\overline{w}^{|v|}}{\beta_v},\quad v\in V,
\end{align}
$($with the convention $0^0=1$$)$  and
 \begin{align}\label{valatk}
\Vsf_w(f)= \is{f}{k_w}_\betab= \sum_{v\in V} f(v)\, w^{|v|},\quad  f \in\ell^2(\betab),
 \end{align} with the series $\sum_{v\in V} f(v)\, w^{|v|}$ being absolutely summable.
\end{thm}
\begin{proof}
If (iv) holds, then the function $k_w\colon V\to\C$ defined by \eqref{altbpe2+} belongs to $\ell^2(\betab)$. Clearly, for every $u\in V$ it satisfies $\is{e_u}{k_w}_\betab=w^{|u|}$, which, by linearity, gives \eqref{altbpe2} and proves (iii). The uniqueness of $k_w$ follows easily from \eqref{altbpe2}.

That (iii) implies (ii) follows immediately from the Cauchy-Schwartz inequality.

If (ii) is satisfied, then the density of $\ee$ in $\ell^2(\betab)$ implies that for every $f \in \ell^2(\betab)$ the series $\sum_{v\in V} f(v) w^{|v|}$ is absolutely summable and $\sum_{ v\in V} |f(v)|\, |w|^{|v|}\Le c\|f\|_\betab$. Thus, the mapping
\begin{align*}
\ee\ni f\mapsto \sum_{v\in V} f(v)\, w^{|v|}\in \C
\end{align*}
can be extended (in a unique way) to a continuous linear mapping $\Vsf_w\colon\ell^2(\betab)\to\C$, which gives (i).

If (i) is satisfied, then by the Riesz theorem there exists $k\in\ell^2(\betab)$ such that $\Vsf_w(f)=\is{f}{k}_\betab$ for all $f\in\ell^2(\betab)$. This and \eqref{altbpe1} yield $k=k_w$, with $k_w$ given by \eqref{altbpe2+}. Hence $k_w\in\ell^2(\betab)$, which is equivalent to (iv). This completes the proof.
\end{proof}
Under the assumption of Theorem \ref{altbpe0}, if $w\in\C$ and any of the conditions (i)--(iv) is satisfied, then we call $w$ a {\em bounded point evaluation on $\tcal_\betab$}. By $\bpe{\tcal_\betab}$ we denote the set of all bounded point evaluations on $\tcal_\betab$.
\begin{cor}\label{cir}
Suppose $\tcal=(V,E)$ is a countably infinite rooted directed tree and $\betab =\{\beta_v\}_{v\in V}\subseteq (0,\infty)$. Then the set $\bpe{\tcal_\betab}$ is circular.
\end{cor}
Throughout the rest of this section we will assume additionally to \eqref{stand2} that $\slam$ is bounded on $\ell^2(\betab)$ and the weights $\lambdab$ are distributed on $\tcal$ in a special way. These assumptions are gathered in the following.
\begin{align} \label{stand3}\tag{${\star\star}$}
   \begin{minipage}{75ex}
$\tcal=(V,E)$ is a countably infinite rooted and leafless directed tree, $\betab =\{\beta_v\}_{v\in V}\subseteq (0,\infty)$ and $\lambdab=\{\lambda_v\}_{v\in V^\circ}\subseteq(0,\infty)$ satisfy
$\sup_{u\in V} \sum_{v\in\dzii(u)}|\lambda_{v}|^2\tfrac{\beta_v}{\beta_u}<\infty$ and $\sum_{v\in\dzi{u}}\lambda_v =1$ for every $u\in V$.
   \end{minipage}
   \end{align}
\begin{prop}\label{bpenorma}
Assume \eqref{stand3}. Let $\slam$ be a weighted shift on $\tcal_\betab$. If $w\in\C$ satisfies $|w|=\|\slam\|$, then $w\notin\bpe{\tcal_\betab}$.
\end{prop}
\begin{proof}
Let $|w|=\|\slam\|$. By Lemma \ref{norma_tk}, for every $u\in V$, $|w|^2 \Ge \sum_{v\in\dzi{u}}\lambda_v^2 \tfrac{\beta_v}{\beta_u}$. Hence, by the Cauchy-Schwarz inequality we have
\begin{align*}
1= \sum_{v\in\dzi{u}}\lambda_v \Le \bigg(\sum_{v\in\dzi{u}}\lambda_v^2 \frac{\beta_v}{\beta_u} \bigg)\bigg(\sum_{v\in\dzi{u}}\frac{\beta_u}{\beta_v}\bigg)\Le \sum_{v\in\dzi{u}}\frac{|w|^2}{\beta_v}\,\beta_u,\quad u\in V.
\end{align*}
This in turn implies that
\begin{align*}
\sum_{v\in\dzi{u}} \frac{|w|^{2|u|+2}}{\beta_v} =\frac{|w|^{2|u|}}{\beta_u}\sum_{v\in\dzi{u}} \frac{|w|^2}{\beta_v}\,\beta_u \Ge\frac{|w|^{2|u|}}{\beta_u},\quad u\in V.
\end{align*}
Thus $$ \sum_{|v| = k+1} \frac{|w|^{2k+2}}{\beta_v} \Ge \sum_{|v| = k} \frac{|w|^{2k}}{\beta_v}, \quad k \in \N,$$ which according to Theorem \ref{altbpe0}\,(iv) proves that $w\not\in\bpe{\tcal_\betab}$.
\end{proof}
\begin{lem}
Assume \eqref{stand3}. Let $\slam$ be a weighted shift on $\tcal_\betab$. If $w\in \bpe{\tcal_\betab}$, then 
\begin{align}\label{altbpe3+}
\Vsf_w(\slam f)= w\Vsf_w(f),\quad f\in\ell^2(\betab).
\end{align}
\end{lem}
\begin{proof}
Since $\sum_{s\in\dzii{(u)}} \lambda_s=1$ for every $u \in V$, by Theorem \ref{altbpe0}, we have\allowdisplaybreaks
\begin{align*}
\Vsf_w(\slam e_u\big)
&=\sum_{v\in V} \big(\slam e_u\big)(v)\, w^{|v|}
=\sum_{v\in V} \Big(\sum_{s\in\dzii{(u)}} \lambda_s e_s \Big)(v)\, w^{|v|}\\
&=\sum_{s\in\dzii{(u)}} \lambda_s  \, w^{|s|}=\bigg(\sum_{s\in\dzii{(u)}} \lambda_s\bigg) \, w^{|u|+1}
=w\, w^{|u|}
= w\,\Vsf_w(e_u),\quad u\in V.
\end{align*}
Hence, applying linearity, we see that the equality in \eqref{altbpe3+} holds with every $f\in\ee$. This, continuity of $\Vsf_w$ and density of $\ee$ in $\ell^2(\betab)$ implies \eqref{altbpe3+}.
\end{proof}
\begin{prop}\label{bpespec}
Assume \eqref{stand3}. Let $\slam$ be a weighted shift on $\tcal_\betab$. Then $\bpe{\tcal_\betab}\subseteq\sigma_p\big(\slam^*\big)$.
\end{prop}
\begin{proof}
Suppose $w\in\bpe{\tcal_\betab}$. Then, by Theorem \ref{altbpe0}, $k_w\in\ell^2(\betab)$. According to  \eqref{valatk} and \eqref{altbpe3+}, we have
\begin{align*}
\is{f}{\slam^*k_w}_\betab&=\is{\slam f}{k_w}_\betab=\Vsf_w(\slam f)\\
&=w \Vsf_w(f)=w \is{f}{k_w}_\betab=\is{f}{\overline{w} k_w}_\betab,\quad f\in \ell^2(\betab).
\end{align*}
Hence $\slam^*k_w=\overline{w}\, k_w$ and as a consequence $\overline{w}\in \sigma_p(\slam^*)$. Applying  Corollary \ref{cir} we get the claim.
\end{proof}
The following example shows that, in contrast to classical weighted shift, the inclusion in theorem above may be proper (see \cite[Theorem 10\,(i)]{Shields}).
\begin{figure}
\begin{center}
\begin{tikzpicture}[scale = .6, transform shape]
\tikzstyle{every node} = [circle,fill=gray!30]
\node (e10)[font=\footnotesize, inner sep = 1pt] at (0,0) {$(0,0)$};

\node (e11)[font=\footnotesize, inner sep = 1pt] at (3,1) {$(1,1)$};
\node (e12)[font=\footnotesize, inner sep = 1pt] at (6,1) {$(1,2)$};
\node (e13)[font=\footnotesize, inner sep = 1pt] at (9,1) {$(1,3)$};
\node[fill = none] (e1n) at(12,1) {};

\node (f11)[font=\footnotesize, inner sep = 1pt] at (3,-1) {$(2,1)$};
\node (f12)[font=\footnotesize, inner sep = 1pt] at (6,-1) {$(2,2)$};
\node (f13)[font=\footnotesize, inner sep = 1pt] at (9,-1) {$(2,3)$};
\node[fill = none] (f1n) at(12,-1) {};

\draw[->] (e10) --(e11) node[pos=0.5,above = 0pt,fill=none] {$1/2$};
\draw[->] (e11) --(e12) node[pos=0.5,above = 0pt,fill=none] {$1$};
\draw[->] (e12) --(e13) node[pos=0.5,above = 0pt,fill=none] {$1$};
\draw[dashed, ->] (e13)--(e1n);
\draw[->] (e10) --(f11) node[pos=0.5,below = 0pt,fill=none] {$1/2$};
\draw[->] (f11) --(f12) node[pos=0.5,below = 0pt,fill=none] {$1$};
\draw[->] (f12) --(f13) node[pos=0.5,below = 0pt,fill=none] {$1$};
\draw[dashed, ->] (f13)--(f1n);
\end{tikzpicture}
\end{center}
\caption{Tree $\tcal_{2,0}$}
\label{drzewot201}
\end{figure}
\begin{exa}
Let $\tcal_{2,0} = (V_{2,0}, E_{2,0})$ be the directed tree defined by (see Figure \ref{drzewot201})
   \allowdisplaybreaks
   \begin{align*}
V_{2,0} & = \{(0,0)\} \cup
\Big\{(i,j)\colon i\in \{1,2\},\, j\in \mathbb{N}\Big\},
   \\
E_{2,0} & =
\Big\{((0,0),(i,1))\colon i \in \{1,2\}\Big\} \cup
\Big\{((i,j),(i,j+1))\colon i\in \{1,2\},\, j\in
\mathbb{N}\Big\}.
\end{align*}

Clearly, $\tcal_{2,0}$ is leafless and rooted. Let $\lambdab=\{\lambda_v\}_{v\in V_{2,0}^\circ}$ and $\betab=\{\beta_v\}_{v\in V_{2,0}}$ be given by
\begin{align*}
\lambda_{(i,j)} =
\left\{
  \begin{array}{ll}
   \tfrac{1}{2}  & \hbox{if $i\in\{1,2\}$ and $j=1$,} \\
   1  &
 \hbox{if $i\in\{1,2\}$ and $j > 1$,}
  \end{array}
\right.
\end{align*}
and
\begin{align*}
\beta_{(i,j)} =
\left\{
     \begin{array}{cl}
     1 & \hbox{if $i=1$ and $j \Ge 1$ or $i=j=0$,} \\
     \tfrac{1}{4^j} & \hbox{if $i=2$ and $j\Ge 1$.}
     \end{array} \right.
\end{align*}
Then $\tcal_{2,0}$, $\betab$ and $\lambdab$ fulfill condition \eqref{stand3}. By Theorem \ref{altbpe0} we get $\bpe{\tcal_{2,0},\betab} =\{ z\colon |z|<\tfrac 12\}$. On the other hand, by Proposition \ref{thm-unit+-} below, the weighted shift $\slam$ on $(\tcal_{2,0},{\boldsymbol\betab})$ is unitarily equivalent to the weighted shift $S_{\boldsymbol \mu}$ on $(\tcal_{2,0},\mathbb{1})$, where $\mub=\{\mu_v\}_{v\in V_{2,0}^\circ}$ is given by
\begin{align*}
\mu_{(i,j)} =
\left\{ \begin{array}{cl}
\tfrac{1}{2} & \hbox{if $i= 1$ and $j=1$,}\\
1 & \hbox{if $i= 1$ and $j >1$,} \\
\tfrac{1}{4} & \hbox{if $i=2$ and $j =1$,} \\
\tfrac{1}{2} & \hbox{if $i=2$ and $j >1$.}
\end{array} \right.
\end{align*}
For $\theta\in\{z\in\C\colon |z|<1\}$, let $k_\theta \colon V_{2,0}\to \C$ be defined by
 \begin{align*}
k_\theta{(i,j)} =
\left\{
  \begin{array}{ll}
   \tfrac{1}{2}  & \hbox{if $i=0$ and $j=0$,} \\
\theta^j  & \hbox{if $i=1$ and $j>0$,}\\
   0  & \hbox{if $i=2$ and $j>0$.}
  \end{array}
\right.
\end{align*}
Then $k_\theta$ is an eigenvector of $S_\mub$ corresponding to an eigenvalue $\theta$. Hence $\{ z\colon |z|<1\} \subseteq \sigma_p(S_{\boldsymbol \mu}^*) = \sigma_p(\slam^*)$. In fact we can show that $\{ z\colon |z|<1\} = \sigma_p(S_{\boldsymbol \mu}^*)$. Indeed, suppose that $\theta\in \{z\in \C\colon |z|\Ge 1\}$ is an eigenvalue of $S_\mub^*$ corresponding to an eigenvector $k\in\ell^2(V,\mathbb{1})$. Then $\big(S_\mub^*\big)^nk=\theta^n k$ for every $n\in\N$, which implies that for every $(i,j)\in V_{2,0}$, $k(i,j)=\alpha_i \theta^j$ with some $\alpha_i\in\C$. This and $k\in\ell^2(V_{2,0},\mathbb{1})$ yield $k=0$, a contradiction.
\end{exa}
According to Proposition \ref{bpenorma} and Proposition \ref{bpespec} we see that $\bpe{\tcal_\betab}\subseteq \varDelta_{r(\slam)}$ whenever $\slam$ is normaloid (i.e., $r(\slam)=\|\slam\|$). If $\slam$ is non-normaloid, then the following question arises.

\begin{prob}\label{brzegbpe}
Does there exists $\hat\varphi\in\multib$ such that $\bpe{\tcal_\betab}=\overline{\varDelta_{r(\slam)}}$ and $\sum_{k=0}^\infty\hat\varphi(k)\,w^k$ is divergent for any $w\in\partial \bpe{\tcal_\betab}$?
\end{prob}
In view of Propositions \ref{altbpe8} and \ref{bpespec} (cf. Problem \ref{brzegbpe}), for any $\hat\varphi\in\multib$ we may define
\begin{align*}
\varphi (w) = \sum_{n\in \N_0} \hat\varphi(n)\, w^n,\quad w\in \wn{\bpe{\tcal_\betab}}.
\end{align*}
Clearly, $\varphi$ is analytic.
\begin{thm}\label{main}
Let $\tcal=(V,E)$ be a countably infinite rooted and leafless directed tree, $\betab =\{\beta_v\}_{v\in V}\subseteq (0,\infty)$ and $\lambdab=\{\lambda_v\}_{v\in V^\circ}\subseteq(0,\infty)$. If $\sum_{v\in\dzi{u}}\lambda_v =1$ for every $u\in V$ and $\slam \in \bsb(\ell^2(\betab))$, then
\begin{align}\label{multw}
\Vsf_w(M_{\hat\varphi} f)= \varphi(w)\,\Vsf_w(f), \qquad f\in  \ell^2(\betab), \ \hat\varphi\in\multib, \ w\in\wn{\bpe{\tcal_\betab}}.
\end{align}
\end{thm}
\begin{proof}
Fix $w\in\wn{\bpe{\tcal_\betab}}$, $\hat\varphi\in\multib$ and $f\in \ell^2(\betab)$. For $n\in\N$, let $\hat\varphi^{(n)}\colon \N_0\to\C$ be given by $\hat\varphi^{(n)}(k)=\chi_{\{1,\ldots,n\}}(k)\, \hat\varphi(k)$. Since $\slam \in \bsb(\ell^2(\betab))$,  $\hat\varphi^{(n)}\in \multib$ by Theorem \ref{cauchym} (ii).
By induction and linearity from  \eqref{altbpe3+} we get
 \begin{align}\label{multn}
\Vsf_w(M_{\hat\varphi^{(n)}} f)= \varphi^{(n)}(w)\,\Vsf_w(f).
\end{align}
Since $w\in\wn{\bpe{\tcal_\betab}}$, the series $\sum_{v\in V} f(v)\,w^{|v|}$ is absolutely summable by Theorem \ref{altbpe0}, whence the series $\sum_{n\in \N_0} \hat\varphi(n)\, w^n$ is absolutely convergent by Proposition \ref{altbpe8}. Therefore we can write
\begin{multline*}
\varphi(w)\,\Vsf_w(f)
= \Big(\sum_{k\in \N_0} \hat\varphi(k)\, w^k\Big)\Big(\sum_{v\in V} f(v)\,w^{|v|}\Big)\\
=\Big(\sum_{k\in \N_0} \hat\varphi(k)\, w^k\Big) \Big(\sum_{l\in\N_0}\sum_{v\in \dzin{l}{\koo}} f(v)\,w^{l}\Big)
=\sum_{m\in\N_0} \alpha_{\hat\varphi}(m)\, w^m,
\end{multline*}
with $\alpha_{\hat\varphi}\colon \N_0\to\C$ given by the Cauchy product formula. In the same manner, we see that
\begin{align*}
\varphi^{(n)}(w)\,\Vsf_w(f)=\Big(\sum_{k\in \N_0} \hat\varphi^{(n)}(k)\, w^k\Big) \Big(\sum_{l\in\N_0}\sum_{v\in \dzin{l}{\koo}} f(v)\,w^{l}\Big)=\sum_{m\in\N_0} \alpha_{\hat\varphi^{(n)}}(m)\, w^m,
\end{align*}
with $\alpha_{\hat\varphi^{(n)}}\colon \N_0\to\C$. It is easily seen that
\begin{align}\label{coef}
\alpha_{\hat\varphi^{(n)}}(m)=\alpha_{\hat\varphi}(m),\quad m\Le n,\, n\in\N_0.
\end{align}
By Theorem \ref{altbpe0}, both the series $\sum_{v\in V} \big(M_{\hat\varphi} f(v)\big)\,w^{|v|}=V_w(M_\varphi f)$ and $\sum_{v\in V} \big(M_{\hat\varphi^{(n)}} f(v)\big)\,w^{|v|}=V_w(M_{\varphi^{(n)}} f)$ are absolutely summable. Moreover, for every $n\in\N_0$ and $v\in V$ such that $|v|\Le n$ we have
 \begin{align*}
 \big( M_{\hat\varphi} f\big)(v)=  \sum_{k=0}^{|v|} \lambda_{\paa^k(v)|v} \ \hat\varphi(k) \,f ({\paa}^k(v))
= \sum_{k=0}^{|v|} \lambda_{\paa^k(v)|v}\ \hat\varphi^{(n)}(k) \,f ({\paa}^k(v))=\big( M_{\hat\varphi^{(n)}} f\big)(v),
\end{align*}
which implies that coefficients of the series
\begin{align*}
\sum_{v\in V} \big(M_{\hat\varphi} f(v)\big)\,w^{|v|}\quad \text{and}\quad\sum_{v\in V} \big(M_{\hat\varphi^{(n)}} f(v)\big)\,w^{|v|}
\end{align*}
at $v\in V$ with $|v|\Le n$ are the same. Combining this, \eqref{coef} and \eqref{multn} gives \eqref{multw}.
\end{proof}
\begin{prop}\label{inkluzja}
Assume \eqref{stand3}. For every $\hat\varphi\in\multib$, $\varphi\big(\wn{\bpe{\tcal_\betab}}\big)^*\subseteq \sigma_p\big(M_{\hat\varphi}^*\big)$, where $\varOmega^*=\{\overline{z}\colon z\in \varOmega\}$.
\end{prop}
\begin{proof}
Fix $w\in\wn{\bpe{\tcal_\betab}}$. Then, by Theorem \ref{altbpe0}, $k_w\in\ell^2(\betab)$. Thus, applying Theorem \ref{altbpe0} and \eqref{multw}, we have
\begin{align*}
\is{f}{M_{\hat\varphi}^*k_w}_\betab&=\is{M_{\hat\varphi} f}{k_w}_\betab=\Vsf_w(M_{\hat\varphi} f)\\
&=\varphi(w) \Vsf_w(f)=\varphi(w) \is{f}{k_w}_\betab=\is{f}{\overline{\varphi(w)} k_w}_\betab,\quad f\in \ell^2(\betab).
\end{align*}
Hence $M_{\hat\varphi}^*k_w=\overline{\varphi(w)}\, k_w$, and as a consequence $\overline{\varphi(w)}\in \sigma_p\big(M_{\hat\varphi}^*\big)$.
\end{proof}
\begin{rem}
As regards Proposition \ref{inkluzja}, it is worth noting that if $\tcal$ is isomorphic to the directed tree $\big(\N_0,\{(n,n+1)\colon n\in\N_0\}\big)$, then  $\varphi\big(\bpe{\tcal_\betab}\big)^*\subseteq\sigma_p\big(M_{\hat\varphi}^*\big)$ (see \cite[Theorem 10]{Shields}).
\end{rem}
\section{Point spectrum of $\slam^*$ via paths}\label{PSVPsec}
In this section we aim to show that some information about point spectrum of $\slam^*$ can be deduced from the behaviour of $\slam^*$ on paths.

We begin with easy observation that the point spectrum of $\slam^*$ contains all the bounded point evaluations calculated with respect to paths, defined in the following natural way: if $\tcal=(V, E)$ is a countably directed tree, $\pth=(V_\pth, E_\pth)$ is a path in $\tcal$ and $\betab=\{\beta_v\}_{v\in V}\subseteq (0,\infty)$, then $\bpe{\pth_\betab}$, the set of {\em bounded point evaluations with respect to $\pth$}, is defined as the set of all $w\in\C$ such that $\sum_{v\in V_\pth}\beta_v^{-1} |w|^{2|v|}<\infty$.
\begin{prop}\label{unionbpe}
Assume \eqref{stand3}. Let $\slam$ be a weighted shift on $\tcal_\betab$. Then the following assertions hold$:$
\begin{enumerate}
\item[(i)] for every $\pth\in\pp$, $\bpe{\pth_\betab}\subseteq \sigma_p\big(\mathcal{Q}_\pth\slam^*|_{\ell^2(\pth,\betab)}\big)$,
\item[(ii)] $\bpe{\tcal_\betab} \subseteq \bigcup_{\pth\in\pp}\bpe{\pth_\betab}\subseteq \sigma_p\big(\slam^*\big)$,
\item[(iii)] if $\card{\pp(\tcal)}<\infty$, then $\bpe{\tcal_\betab}= \bigcup_{\pth\in\pp}\bpe{\pth_\betab}$.
\end{enumerate}
\end{prop}
\begin{proof}
Arguing as in the proof of Proposition \ref{bpespec} we show that (i). The first inclusion in assertion (ii) and assertion (iii) follow immediately from the definitions of $\bpe{\tcal_\betab}$ and $\bpe{\pth_\betab}$, $\pth\in\pp$. Then, the rest follows easily from the fact that for every $\pth=(V_\pth,E_\pth)\in\pp$, $\sigma_p\big(\mathcal{Q}_\pth\slam^*|_{\ell^2(\pth,\betab)}\big)\subseteq \sigma_p\big(\slam^*\big)$ (see Lemma \ref{adjoint+invar}).
\end{proof}
Another and more concrete way of localizing the point spectrum of $\slam$ is available. One definition more is required: assuming \eqref{stand2}, we set
\begin{align*}
r_2^\pth(\slam) &= \lim_{n \to \infty} \inf\bigg\{ \Big(\sqrt{\beta_v}\,\lambda_{\koo|v}^\pth\Big)^{\frac{1}{|v|}} \colon v \in \pth, \ |v| \geqslant n \bigg\},\quad \pth\in\pp,\\
r_2^+(\slam) &= \sup \Big\{r_2^\pth(\slam)\colon \pth\in\pp\Big\}.
\end{align*}
\begin{thm}\label{r2+widmo}
Let $\tcal=(V,E)$ be a countably infinite rooted and leafless directed tree, $\betab=\{\beta_v\}_{v\in V}\subseteq (0,\infty)$ and $\{\lambda_v\}_{v \in V^\circ}\subseteq  (0,\infty)$. Let $\slam\in\bsb(\ell^2(\betab))$ be a weighted shift on $\tcal_\betab$. Then $\varDelta_{r_2^+(\slam)}\subseteq \sigma_p\big(\slam^*\big)$.
\end{thm}
\begin{proof}
Fix $\pth\in\pp$. Then for every $k\in\N_0$ there exists unique $v_k\in\pth$ such that $|v_k|=k$. Define the sequence $\mub=\{\mu_k\}_{k=0}^\infty\subseteq(0,\infty)$ by $\mu_k=\sqrt{\frac{\beta_{v_{k+1}}}{\beta_{v_k}}}\,\lambda_{v_{k+1}}$. Let $S_{\mub}$ be the classical weighted shift on $\ell^2(\N_0)$ with weights $\mub$. Then one can easily show using Lemma \ref{adjoint+invar} that $S_\mub^*$ is unitarily equivalent to $\mathcal{Q}_\pth\slam^*|_{\ell^2(\pth,\betab)}$ via $U\colon \ell^2(\N_0)\to\ell^2(\pth,\betab)$ given by $(Uf)(u)=\beta_u^{-1/2} f(u)$ for $u\in V$, $f\in\ell^2(\N_0)$. According to \cite[Theorem 8]{Shields}, we have
\begin{align} \label{inclwidmo}
\varDelta_{r_2(\pth)}\, \subseteq\, \sigma_p(S_{\mub}^*)\,\subseteq \, \overline{\varDelta_{r_2(\pth)}} ,
\end{align}
with $r_2(\pth)=\liminf_{k\to\infty} \big(\mu_{0}\cdots\mu_{k}\big)^\frac{1}{k+1}$. Clearly, $r_2^\pth(\slam)=r_2(\pth)$. This, \eqref{inclwidmo} and Lemma \ref{adjoint+invar} imply that
\begin{align*}
\varDelta_{r_2^+(\slam)}\subseteq \bigcup_{\pth\in\pp} \sigma_p\big(\mathcal{Q}_{\pth}\slam^*|_{\ell^2(\pth,\betab)}\big)\subseteq \sigma_p(\slam^*),
\end{align*}
which completes the proof.
\end{proof}
Using the unitary equivalence of operators  $\mathcal{Q}_{\pth}\slam^* |_{\ell^2(\pth,\betab)}$ and $S^*_{\mub}$ and applying \cite[Proposition 7 and Theorem 10\,(i)]{Shields} we get
\begin{cor}\label{pathbpesig}
Assume \eqref{stand2}. Let $\slam\in\bsb(\ell^2(\betab))$ be a weighted shift on $\tcal_\betab$. Then $\varDelta_{r_2^\pth(\slam)}=\sigma_p\big(\mathcal{Q}_{\pth}\slam^*|_{\ell^2(\pth,\betab)}\big)$ for all $\pth$.
\end{cor}
In view of \cite[Theorem 8]{Shields}, it is reasonable to ask for a solution of the following.
\begin{prob}\label{probr2}
When does the inclusion
\begin{align*}
\sigma_p\big(\slam^*\big)\subseteq  \overline{\varDelta_{r_2^+(\slam)}}
\end{align*}
hold?
\end{prob}
As shown below the inclusion of Problem \ref{probr2} is satisfied if the tree has finitely many branching vertexes. Recall that a vertex $v$ is called a {\em branching vertex} if the set $\dzi{v}$ contains at least two distinct vertices.
\begin{prop}
Assume \eqref{stand3}. If the directed tree $\tcal$ has finitely many branching points, then
$\sigma_p\big(\slam^*\big) \subseteq \overline{\Delta_{r_2^+(\slam)}}$.
\end{prop}
\begin{proof}
Since $\tcal$ has finite number of branching vertices, there is $n_0\in\N$ such that $\dzi{\paa(v)}=\{v\}$ for every $v\in V$ with $|v|\Ge n_0$. Let $w\in \sigma_p(\slam^*)$ and let $f\in\ell^2(\betab)$ be the corresponding eigenvector, i.e.,
\begin{align} \label{ker}
wf(u)=  \sum_{v\in\dzi{u}} \lambda_v\;\frac{\beta_v}{\beta_u}\;f(v),\quad u\in V.
\end{align}
We claim that for every $\pth\in\pp$ there is $\alpha_\pth\in\C$ such that
\begin{align}\label{ogon}
f(v)=\alpha_\pth \frac{w^{|v|}}{\beta_v},\quad v\in \{ u\in \pth\colon |u|\Ge n_0\}.
\end{align}
Indeed, fix $\pth\in\pp$. Let $v_\pth$ be  the only vertex in $\pth$ such that $|v_\pth|=n_0$ and  let $v_k$, $k\in\N$, be the unique vertex in $\pth$ such that $|v_k|=n_0+k$. In view of \eqref{stand3}, $\lambda_{v_k}=1$ for every $k\in \N$. Then, applying \eqref{ker} repeatedly $k$-times, we deduce that
\begin{align*}
f(v_k)= w^k \frac{\beta_{v_\pth}}{\beta_{v_k}}\;  f(v_\pth) ,\quad k\in\N,
\end{align*}
which clearly gives \eqref{ogon} with $\alpha_\pth=w^{-|v_\pth|}\beta_{v_\pth}f(v_\pth)$.

Since $f$ is non-zero, there is $\tilde\pth\in\pp$ such that $f(v_{\tilde\pth})\neq 0$. In view of \eqref{ogon}, the inequalities
\begin{align*}
\alpha_{\tilde\pth}^2\sum_{k=1}^\infty\bigg(\frac{w^{k}}{\beta_{v_k}}\bigg)^2\beta_{v_k}\Le \sum_{v\in V} |f(v)|^2\beta_v<\infty,
\end{align*}
Theorem \ref{altbpe0} and Proposition \ref{unionbpe} we have $w\in\bpe{\pth_\betab}\subseteq\sigma_{p}\big(\mathcal{Q}_\pth\slam^*|_{\ell^2(\tilde\pth,\betab)}\big)=\sigma_p(S^*_{\mub})$, where $S_{\mub}$ is as in the proof of Theorem \ref{r2+widmo}. Hence, by \eqref{inclwidmo} we have $|w|\Le r_2^{\tilde\pth}(\slam)\Le r_2^+(\slam)$. This proves the claim.
\end{proof}
On the other hand, if there is plenty of branching vertexes in each generation, then the inclusion of Problem \ref{probr2} does not hold. This follows from the Proposition \ref{k-ary} below. Recall that for $\kappa\in\N$, the rooted {\em $\kappa$-ary directed tree} is the directed tree $\tcal^{(\kappa)}=\big(V^{(\kappa)}, E^{(\kappa)}\big)$ given by
\begin{align*}
V^{(\kappa)} &= \{ (k,l)\colon k \in \N_0 \text{ and } l \in J_{\kappa^k}\},\\
E^{(\kappa)} &= \Big\{ \big( (k,l), (k+1,m) \big) \colon k \in \N_0, l \in J_{\kappa^k}, m=\kappa( l-1)+1,\ldots, \kappa\, l \Big\}.
\end{align*}
Figure \ref{Fig1} above gives a more intuitive description of $\kappa$-ary directed trees.
\hspace*{0em}
{
\begin{figure}
\begin{tikzpicture}[scale = .6, transform shape, line cap=round,line join=round,x=1.2cm,y=1cm]
   \tikzstyle{every node} = [circle, fill=gray!30]

   \node (v01)[font=\footnotesize, inner sep = 1pt]at (0,0){(0,1)};

   \node (v11)[font=\footnotesize, inner sep = 1pt] at (5,2){(1,1)};
   \node (v12)[font=\footnotesize, inner sep = 1pt] at (5,0){(1,2)};
   \node (v13)[font=\scriptsize, inner sep = 1pt] at (5,-2){(1,3)};

   \node (v21)[font=\footnotesize, inner sep = 1pt] at (10,4.25){(2,1)};
   \node (v22)[font=\footnotesize, inner sep = 1pt] at (10,3.25){(2,2)};
   \node (v23)[font=\footnotesize, inner sep = 1pt] at (10,2.25){(2,3)};

   \node (v24)[font=\footnotesize, inner sep = 1pt] at (10,1){(2,4)};
   \node (v25)[font=\footnotesize, inner sep = 1pt] at (10,0){(2,5)};
   \node (v26)[font=\footnotesize, inner sep = 1pt] at (10,-1){(2,6)};

   \node (v27)[font=\footnotesize, inner sep = 1pt] at (10,-2.25){(2,7)};
   \node (v28)[font=\footnotesize, inner sep = 1pt] at (10,-3.25){(2,8)};
   \node (v29)[font=\footnotesize, inner sep = 1pt] at (10,-4.25){(2,9)};

   \node (v313)[font=\footnotesize, inner sep = 0.25pt] at (15,1){(3,13)};
   \node (v314)[font=\footnotesize, inner sep = 0.25pt] at (15,0){(3,14)};
   \node (v315)[font=\footnotesize, inner sep = 0.25pt] at (15, -1){(3,15)};

   \draw[->] (v01) --(v11) node[pos=0.6,below = -5pt,fill=none, font = \footnotesize] {$1/3$};
   \draw[->] (v01) --(v12) node[pos=0.6,below = -4pt,fill=none, font = \footnotesize] {$1/3$};
   \draw[->] (v01) --(v13) node[pos=0.6,below =  -3pt,fill=none, font = \footnotesize] {$1/3$};

   \draw[->] (v11) --(v21) node[pos=0.6,below = -5pt,fill=none, font = \footnotesize] {$1/3$};
   \draw[->] (v11) --(v22) node[pos=0.6,below = -4pt,fill=none, font = \footnotesize] {$1/3$};
   \draw[->] (v11) --(v23) node[pos=0.6,below =  -3pt,fill=none, font = \footnotesize] {$1/3$};

   \draw[->] (v12) --(v24) node[pos=0.6,below = -5pt,fill=none, font = \footnotesize] {$1/3$};
   \draw[->] (v12) --(v25) node[pos=0.6,below = -4pt,fill=none, font = \footnotesize] {$1/3$};
   \draw[->] (v12) --(v26) node[pos=0.6,below = -3pt,fill=none, font = \footnotesize] {$1/3$};

   \draw[->] (v13) --(v27) node[pos=0.6,below = -5pt,fill=none, font = \footnotesize] {$1/3$};
   \draw[->] (v13) --(v28) node[pos=0.6,below = -4pt,fill=none, font = \footnotesize] {$1/3$};
   \draw[->] (v13) --(v29) node[pos=0.6,below = -3pt,fill=none, font = \footnotesize] {$1/3$};

   \draw[->] (v25) --(v313) node[pos=0.6,below = -5pt,fill=none, font = \footnotesize] {$1/3$};
   \draw[->] (v25) --(v314) node[pos=0.6,below = -4pt,fill=none, font = \footnotesize] {$1/3$};
   \draw[->] (v25) --(v315) node[pos=0.6,below = -3pt,fill=none, font = \footnotesize] {$1/3$};
\end{tikzpicture}
\caption{3-ary directed tree}
\label{Fig1}
\end{figure}
}
\begin{prop}\label{k-ary}
Let $\kappa\in\N$. Let $\slam$ be the weighted shift on $\tcal^{(\kappa)}_\betab$ with $\lambda_v=\kappa^{-1}$ for $v\in V^{(\kappa)}\setminus\{(0,1)\}$ and $\beta_v=\kappa^{-|v|}$ for all $v\in V^{(\kappa)}$. Then  $\slam\in\bsb(\ell^2(\betab))$, $r_2^+(\slam)=\frac{1}{\sqrt{\kappa^3}}$ and $\{z\in\C\colon |z|<\frac1\kappa\}\subseteq \sigma_p(\slam^*)$.
\end{prop}
\begin{proof} Direct calculation shows that $\slam\in\bsb(\ell^2(\betab))$ and $r_2^+(\slam)=\frac{1}{\sqrt{\kappa^3}}$. Now, for $w\in\C$, let us define the function $f_w\colon V\to \C$ by $f_w(v)= w^{|v|}$  for $v\in V^{(\kappa)}$.
Clearly,
\begin{align*}
\|f_w\|_\mathbb{\betab}^2 = \sum_{v\in V^{(\kappa)}}  \frac{|w|^{2|v|}}{\kappa^{|v|}} =\sum_{n=0}^{\infty} {|w|^{2n}}<\infty\quad  \text{ for }\quad |w|<{1},
\end{align*}
and thus (see Lemma \ref{adjoint+invar})
\begin{align*}
(\slam^*f_w)(u) = \sum_{v \in \dzi{u}} \frac {1}{\kappa^2} f_w(v) = \frac{w^{|u|+1}}{\kappa} = \frac{w}{\kappa} f_w(u) \quad  \text{ for }u\in V^{(\kappa)} \text{ and } |w|<{1}.
\end{align*}
This proves the required inclusion.
\end{proof}
\section{Non-weighted directed trees}\label{Last}
In this section we concentrate on weighted shifts on directed trees $\tcal_{\mathbb{1}}$ and comment on how the results of the previous sections apply in this context.

We begin by showing that a weighted shift $S_\mub$ on $\tcal_{\mathbb{1}}$ is unitarily equivalent to a weighted shift $\slam$ on $\tcal_\betab$ with any given weights $\lambdab$. The proof is essentially extracted from \cite[Example 4.3.1]{2012-j-j-s-jfa}.
\begin{prop} \label{thm-unit+-}
Let $\tcal=(V,E)$ be a countably infinite rooted and leafless directed tree and let $\mub=\{\mu_v\}_{v \in V^{\circ}}$ be a family of nonzero complex numbers. Let $\lambdab=\{\lambda_v\}_{v \in V^{\circ}}\subseteq (0,\infty)$. Define the family $\betab=\{\beta_v\}_{v\in V}\subseteq (0,\infty)$ by
\begin{align*}
\beta_v=\bigg|\frac{\mu_{\koo|v}}{\lambda_{\koo|v}}\bigg|^2,\quad v\in V.
\end{align*}
Then the operator $U \colon \ell^2(V, \mathbb{1}) \to \ell^2(V, \betab)$, defined by
\begin{align*}
\big(U f\big)(u) =
   \frac{\lambda_{\koo|u}}{\mu_{\koo|u}}\, f(u),\quad u\in V,\ f\in \ell^2(\mathbb{1}).
\end{align*}
is unitary and the weighted shift $S_{\mub}$ on $\tcal_{\mathbb{1}}$ is unitarily equivalent to the weighted shift $\slam$  on $\tcal_{\betab}$ via $U$, i.e.,
\begin{align*}
U S_{\mub} = \slam U.
\end{align*}
\end{prop}
\begin{proof}
Since, for every $f\in \ell^2(\mathbb{1})$ we have
\begin{align*}
\sum_{u \in V} |(U f)(u)|^2 \beta_u = \sum_{u \in V} \bigg|\frac{\lambda_{\koo|u}}{\mu_{\koo|u}}\bigg|^2 |f(u)|^2 \beta_u = \sum_{u \in V} |f(u)|^2,
\end{align*}
we see that $U$ is well-defined unitary isomorphism.

Let $f\in\ell^2(\mathbb{1})$ be such that $\varLambda_\tcal^\mub f\in\ell^2(\mathbb{1})$. Thus we get
\begin{multline}
\big(U  \varLambda_\tcal^\mub f\big)(u)
=\frac{\lambda_{\koo|u}}{\mu_{\koo|u}} \, \big(\varLambda_\tcal^\mub f\big)(u)
= \frac{\lambda_{\koo|u}}{\mu_{\koo|u}} \, \mu_u f(\paa(u))\\
=\lambda_u\, \frac{\lambda_{\koo|\paa(u)}}{\mu_{\koo|\paa(u)}} \, f(\paa(u))
= \lambda_u\, (U f)(\paa(u))=\big(\varLambda_{\tcal}^{\lambdab}Uf\big)(u), \quad u\in V\setminus \{\koo\}.\label{unit1+-}
\end{multline}
On the other hand, $\big(U \varLambda_\tcal^{\mub} f\big)(\koo)=0$. This and \eqref{unit1+-} imply that $U f\in \dz{S_{\lambdab}}$. Moreover, $ U S_{\mub} f=\slam U f$. Therefore, we obtain the inclusion $U S_{\mub} \subseteq \slam U$.

Now, suppose that $f\in\ell^2(\mathbb{1})$ satisfy $U f\in\dz{\slam}$. Set $g(u)=\frac{\lambda_{\koo|u}}{\mu_{\koo|u}}\big(\varLambda_\tcal^\mub f\big)(u)$ for $u\in V$. Then, by \eqref{unit1+-}, $g\in\ell^2(V, \betab)$ and consequently $(U^{-1} g)(u)=(\varLambda_\tcal^\mub f)(u)$ for every $u\in V$. This means that $f\in\dz{S_\mub}$. Combining this with the inclusion $U S_{\mub} \subseteq \slam  U$ we complete the proof.
\end{proof}
In Sections \ref{BPEsec} and \ref{PSVPsec} we considered weighted shifts with weights summing over children of every vertex up to $1$ (see condition \eqref{stand3}). In view of the following Proposition we didn't lose much of generality doing so.
\begin{prop}\label{unilat}
Let $\tcal=(V,E)$ be a countably infinite rooted and leafless directed tree and $\mub=\{\mu_v\}_{v \in V^{\circ}}\subseteq \C\setminus\{0\}$. Suppose
\begin{align*}
\mu_{[v]} := \sum_{u\in \dzii(\paa(v))} |\mu_u|^2 <\infty, \quad v \in V^\circ.
\end{align*}
Define the families $[\mub] := \big\{ [\mu]_v\big\}_{v \in V^\circ}$ and $\langle\mub\rangle:= \{ \langle\mu\rangle_v\}_{v \in V}$ by
\begin{align*}
[\mu]_v= \frac{|\mu_v|^2}{\mu_{[v]}}, \quad \langle\mu\rangle_v= \bigg|\frac{\mu_{[\koo |v]}}{\mu_{\koo|v}}\bigg|^2
\quad\text{ with }\quad
\mu_{[u|v]} &=
\left\{ \begin{array}{cl}
1 & \text{if } u=v, \\
\prod_{j=0}^{n-1} \mu_{[\paa^j(v)]} & \text{if }  v \in \dzii^{\langle n\rangle}(u).
\end{array} \right.
\end{align*}
Then the weighted shift $S_{[\mub]}$ on $\tcal_{\langle\mub\rangle}$ is unitarily equivalent to the weighted shift $S_\mub$ on $\tcal_{\mathbb{1}}$. Moreover, for every $u \in V$, $e_u \in \dz{S_{[\mub]}}$ and $S_{[\mub]} e_u   = \sum_{v \in \dzii(u)} \frac{|\mu_v|^2}{\mu_{[v]}} e_v$.
\end{prop}
\begin{proof}
Apply Lemma \ref{thm-unit+-} and \cite[Proposition 3.1.3]{2012-j-j-s-mams}.
\end{proof}
In view of the above, when investigating weighted shifts $S_\mu$ on $\tcal_{\mathbb{1}}$ we can always pass on to $S_{[\mub]}$ on $\tcal_{\langle\mub\rangle}$ whenever a property we are interested in is invariant under unitary equivalence. Note that, assuming $S_\mub$ is bounded, the families $\lambdab:=[\mub]$ and $\betab:=\langle\mub\rangle$ satisfy \eqref{stand3}. This enables us to apply most of the results of previous sections.
\begin{rem}\label{unitarne}
As regards Proposition \ref{unilat} it is worth also to notice that any weighted shift $S_{\lambdab}$ on a weighted directed tree $\tcal_{\betab}$ is unitarily equivalent to a weighted shift $S_\mub$ on a directed tree $\tcal_{\mathbb{1}}$. The equivalence is given by the operator $W \colon \ell^2(\betab) \to \ell^2(\mathbb{1})$, defined by
\begin{align*}
\big(W f\big)(u) =
   \sqrt{\beta_u} f(u),\quad u\in V,\ f\in \ell^2(\betab),
\end{align*}
whence the formula for weights $\mub$ is
\begin{align*}
\mub_v = \sqrt{\frac{\beta_v}{\beta_{\pa{v}}}}\, \lambda_v, \quad v \in V^\circ.
\end{align*}
\end{rem}
Now we observe that, keeping the notation from Proposition \ref{thm-unit+-}, the multiplication operators $M_{\hat\varphi}^{\lambdab,\betab}$ and $M_{\hat\varphi}^{\mub,\mathbb{1}}$ related to $\slam$ and $S_\mub$, respectively, are unitarily equivalent via the operator $U$.
\begin{prop}\label{multikulti}
Under the assumptions of Proposition \ref{thm-unit+-} the following condition is satisfied
\begin{align*}
\Big(\Upsilon\, \Gamma_{\hat\varphi}^\lambdab f\Big)(v)= \Big(\Gamma_{\hat\varphi}^\mub \Upsilon f\Big)(v),\quad v\in V,\ f\in \C^V,
\end{align*}
where $\Upsilon \colon \C^V \to \C^V$ be given by
\begin{align*}
\big(\Upsilon f\big)(u) =
   \frac{\mu_{\koo|u}}{\lambda_{\koo|u}}\, f(u),\quad u\in V,\ f\in \C^V.
\end{align*}
Moreover, $\Upsilon$ induces operator $U^* \colon \ell^2(\betab) \to \ell^2(\mathbb{1})$ and
\begin{align*}
U^* M_{\hat\varphi}^{\lambdab,\betab} = M_{\hat\varphi}^{\mub,\mathbb{1}} U^*.
\end{align*}
\end{prop}
By applying Propositions \ref{thm-unit+-} and \ref{multikulti}, we deduce that $\slam$ and $S_\mub$ induce the same multiplier algebra.

\begin{cor}\label{multalg}
Under the assumptions of Proposition \ref{thm-unit+-}, the algebras ${\mathcal{M}(\mub, \tcal_\betab)}$ and ${\mathcal{M}(\lambdab, \tcal_{\mathbb{1}})}$ coincide.
\end{cor}
\begin{rem}\label{smierdzi}
Under the assumptions of Proposition \ref{thm-unit+-}, the operators $M_{\hat\varphi}^{\lambdab,\betab}$ and $M_{\hat\varphi}^{|\lambdab|,\betab}$, where $|\lambdab| = \{ |\lambda_v|\}_{v \in V^\circ}$, are unitarily equivalent.
\end{rem}
As a consequence of Corollary \ref{multalg} and the unitary equivalence between $\slam$ and $S_\mub$ one can easily derive counterparts of all the results of Section \ref{MOsec} for $S_\mub$ and $\mathcal{M}(\mub,\mathbb{1})$. In particular, in view of Proposition \ref{unilat}, we have
\begin{align}\label{thor}
{\mathcal{M}(\mub, \tcal_{\mathbb{1}})}=\mathcal{M}\Big([\mub], \tcal_{\langle\mub\rangle}\Big).
\end{align}
As regards Section \ref{BPEsec} we see that replacing $\betab$ by $\langle\mub\rangle$ in Theorem \ref{altbpe0} we get
\begin{align*}
\bpe{\tcal_{\langle\mub\rangle}}=\bigg\{w\in\C\colon \sum_{v\in V} \Big|\frac{\mu_{\koo|v}}{\mu_{[\koo|v]}}\Big|^2 \,|w|^{2|v|}<\infty\bigg\}.
\end{align*}
Combining this with Proposition \ref{multikulti} and Corollary \ref{multalg} (with $[\mub]$ and $\langle\mub\rangle$ in place of $\lambdab$ and $\betab$, respectively) we deduce the following
\begin{cor}
Under the assumptions of Proposition \ref{thm-unit+-}, if $S_{\mub} \in \bsb(\ell^2(\mathbb{1}))$, then
\begin{align*}
\widetilde\Vsf_w(M_{\hat\varphi}^{\mub,\mathbb{1}} f)= \varphi(w)\,\widetilde\Vsf_w(f), \qquad f\in  \ell^2(\mathbb{1}),  \ \hat\varphi\in \mathcal{M}(\mub,\tcal_{\mathbb{1}}), \ w\in \wn{\bpe{\tcal_{\langle\mub\rangle}}},
\end{align*}
where $\widetilde\Vsf_w\colon \ell^2(\mathbb{1})\to \C$ is given by $\widetilde\Vsf_w =\Vsf_w\circ U$.
\end{cor}
Using \eqref{thor} and Proposition \ref{inkluzja} we get the following.
\begin{cor}
Under the assumptions of Proposition \ref{thm-unit+-}, if $S_{\mub} \in \bsb(\ell^2(\mathbb{1}))$, then
\begin{align*}
\varphi\Big(\wn{\bpe{\tcal_{\langle\mub\rangle}}}\Big)^*\subseteq \sigma_p\bigg(\Big(M_{\hat\varphi}^{\mub,\mathbb{1}}\Big)^*\bigg), \quad \hat\varphi\in \mathcal{M}(\mub,\tcal_{\mathbb{1}}).
\end{align*}
\end{cor}
Regarding Section \ref{PSVPsec} we note that all the results apply to the weighted shift $S_{\mub}$ on $\tcal_{\mathbb{1}}$.

\section*{Acknowledgments}
We would like to express our sincere gratitude to Professors Jan Stochel and Franciszek Hugon Szafraniec for their valuable comments and for bringing some relevant bibliography items to our attention. In particular, we thank for informing us about the very recent paper \cite{cha-tri} on analytic models for left-invertible weighted shifts on directed trees with finite branching index. Moreover, we thank Professor Jan Stochel for pointing out a possibility of extending some of the results of Section \ref{MOsec} to a context of reproducing kernel Hilbert spaces.

The models \cite{cha-tri}, obtained independently, rely on the work \cite{2001-shi} and enable, under additional assumptions, representing a weighted shift on a directed tree as a multiplication operator $M_z$ on a reproducing kernel Hilbert space $\mathcal{H}$ of $E$-valued holomorphic functions on a disc centered at the origin, where $E=\ker \slam^*$.
\bibliographystyle{amsalpha}

\end{document}